\documentclass[12pt]{amsart}
\usepackage{pifont}
\usepackage{mathrsfs}
\usepackage{geometry}
\usepackage{titletoc}
\usepackage{stix2}

\usepackage{amsmath}
\usepackage{amssymb} 
\usepackage{enumitem} 
\usepackage{mathtools} 
\usepackage[table]{xcolor} 
\usepackage[all]{xy} 
\usepackage{tikz} 
\usepackage{tikz-cd}
\usepackage{indentfirst} 
\usepackage{babel} 
\usepackage{setspace} 

\usepackage[colorlinks,linkcolor=red,anchorcolor=green,citecolor=blue]{hyperref} 
\hypersetup{linktocpage = true} 

\usepackage{rotating} 

\usepackage{ytableau} 
\usepackage{longtable} 
\newcolumntype{M}[1]{>{\centering\arraybackslash}m{#1}} 

\usepackage{fancyhdr}

\newcommand\bp{{\bar\partial}}

\theoremstyle{plain}
\newtheorem{thm}{Theorem}[section]
\newtheorem{lemma}[thm]{Lemma}
\newtheorem{prop}[thm]{Proposition}
\newtheorem{cor}[thm]{Corollary}
\newtheorem{defn}[thm]{Definition}

\theoremstyle{definition}
\newtheorem{example}[thm]{Example}
\newtheorem{remark}[thm]{Remark}

\newcommand{\btheorem}{\begin{thm}}
    \newcommand{\etheorem}{\end{thm}}
\newcommand{\bproposition}{\begin{prop}}
    \newcommand{\eproposition}{\end{prop}}
\newcommand{\bdefinition}{\begin{defn}}
    \newcommand{\edefinition}{\end{defn}}
\newcommand{\bcorollary}{\begin{cor}}
    \newcommand{\ecorollary}{\end{cor}}
\newcommand{\bproof}{\begin{proof}}
    \newcommand{\eproof}{\end{proof}}
\newcommand{\bremark}{\begin{remark}}
    \newcommand{\eremark}{\end{remark}}
\newcommand{\eexample}{\end{example}}
\newcommand{\bexample}{\begin{example}}

\newcommand{\elemma}{\end{lemma}}
\newcommand{\blemma}{\begin{lemma}}

\newcommand{\la}{\langle}
\newcommand{\ra}{\rangle}
\newcommand{\sq}{\sqrt{-1}}

\newcommand{\p}{\partial}

\renewcommand{\bar}{\overline}

\renewcommand{\phi}{\varphi}

\newcommand{\beq}{\begin{equation}}
\newcommand{\eeq}{\end{equation}}
\newcommand{\ee}{\end{eqnarray*}}
\newcommand{\be}{\begin{eqnarray*}}

\newcommand{\bd}{\begin{enumerate}}
    \newcommand{\ed}{\end{enumerate}}

\renewcommand{\hat}{\widehat}

\renewcommand{\bp}{\bar{\partial}}

\newcommand{\C}{{\mathbb C}}

\renewcommand{\P}{{\mathbb P}}

\newcommand{\R}{{\mathbb R}}

\usepackage{harmony}

\renewcommand{\l}{\ell}
\renewcommand{\#}{\sharp}
\newcommand{\pdd}{\left.\frac{d}{dt}\right|_{t=0}}
\newcommand{\vphi}{\varphi}

\newcommand{\lrp}[1]{\left( #1\right)}

\newcommand{\tr}{\mathrm{tr}}

\setlist[itemize]{leftmargin=*}
\setlist[enumerate]{leftmargin=*}

\numberwithin{equation}{section} 

\makeatletter

\usepackage{fancyhdr}
\renewcommand{\leftmark}{{\scriptsize Conformal extremal metrics and constant  scalar curvature}}

\title{Conformal extremal metrics and constant  scalar curvature}

\author{Xiaokui Yang}
\author{Kaijie Zhang}

\address{Xiaokui Yang, Department of Mathematics and Yau Mathematical Sciences Center, Tsinghua University, Beijing, 100084, China}
\email{xkyang@mail.tsinghua.edu.cn}

\address{Kaijie Zhang,  Qiuzhen College, Tsinghua University, Beijing, 100084, China}
\email{zhang-kj21@mails.tsinghua.edu.cn}

\begin{document}

    \begin{abstract} Let $M$ be a compact complex manifold of dimension $n\geq 2$. We prove that for any Hermitian metric $\omega$ on $M$, there exists a unique smooth function $f$ (up to additive constants) such that the conformal metric $\omega_g =e^f \omega$ solves the fourth-order nonlinear PDE
        $$\square_g^*(s_g|s_g|^{n-2})=0,$$
        where $s_g$ is the Chern scalar curvature of $\omega_g$, and $\square_g^*$ denotes the formal adjoint of the complex Laplacian $\square_g=\mathrm{tr}_{\omega_g}\sq\p\bar\p$ with respect to $\omega_g$. This equation arises as the Euler-Lagrange equation of the $n$-Calabi functional
        $$C_{n}(\omega_g)=\int |s_g|^n\frac{\omega_g^n}{n!}$$ within the conformal class of $\omega_g$.    Moreover, we show that  the critical metric $\omega_g$ minimizes the $n$-Calabi functional within the conformal class $[\omega]$. In particular,
        if   $\omega_g$ is a Gauduchon metric, then $\omega_g$ has constant Chern scalar curvature.
    \end{abstract}

    \maketitle

    \section{Introduction}

Let $(M,\omega_g)$ be a compact K\"ahler manifold.  The \textit{Calabi functional} is defined as
    $$
    C(\omega_g)=\int s_g^2\frac{\omega_g^n}{n!},$$
    where $s_g$ is the scalar curvature of $\omega_g$.
It is well-known that extremal K\"ahler metrics, introduced by Calabi in the 1980s (\cite{Calabi1982}),
are critical points of the Calabi functional within the class $[\omega_g]\in H^{1,1}_{\bp}(M,\R)$.
A key characterization is that a K\"ahler metric is extremal
if and only if the gradient of its scalar curvature is a holomorphic vector field.
When the manifold admits no nontrivial holomorphic vector fields, extremal metrics
 are exactly K\"ahler metrics with constant scalar curvature (cscK).
 The foundational work of \cite{ACGC2008} further characterizes extremal metrics
 through Hamiltonian 2-forms and stability, linking their existence to geometric and algebraic invariants.\\

    The search for extremal metrics is deeply intertwined with geometric stability.
    For cscK metrics, the Yau-Tian-Donaldson conjecture establishes an equivalence
    between their existence and the algebro-geometric notion of K-stability.
    While extremal metrics similarly relate to a modified stability condition,
    their existence also depends on the vanishing of obstructions such as the Futaki
    invariant (\cite{Futaki1983}), which measures the imbalance of the
    manifold’s automorphism group. For extremal metrics,
    this invariant must vanish along certain holomorphic vector fields, generalizing the
    cscK case. Calabi conjectured that extremal metrics exist precisely when such
    obstructions are overcome, though a full resolution remains open. Examples of extremal metrics abound in symmetric spaces, such as Fubini-Study metrics on projective varieties. However, non-cscK extremal metrics often arise on manifolds with continuous symmetries, such as Hirzebruch surfaces or blow-ups of  $\C\P^2$, as studied in the context of self-dual Einstein-Hermitian structures  and almost-Hermitian geometry (\cite{ Apostolov1999, Apostolov1996, AP09}). These cases illustrate how extremal metrics balance curvature variations with the underlying complex structure, offering a richer landscape than cscK metrics.  Analytically, extremal metrics satisfy a fourth-order nonlinear PDE, making their study more challenging than the second-order Einstein equations. Progress often relies on perturbative methods, geometric flows, or toric geometry techniques in the presence of symmetry. For this comprehensive topic, we refer to \cite{Calabi1985, Chen2009, Chen-Cheng2021a, Chen-Cheng2021b, CDS15, ChenHe2008,Chen-Sun2014, Don99, Don01, Donaldson2008, Futaki1988, LS94, LWZ2018, Ross06, RT06, Stoppa08, Szekelyhidi2007, Szekelyhidi2014, Tian2002, TW2007} and the references therein.\\

    Recent advancements extend these ideas to non-Kähler settings. Works such as \cite{Angella2017, Angella2020, Lejmi2018} address the Chern-Yamabe problem and Chern-Einstein metrics on Hermitian manifolds, while \cite{BL23} explores second Chern-Einstein metrics on almost-Hermitian 4-manifolds. The interplay between Gauduchon curvature and Hermitian geometry, as in \cite{Broder2023, Fu2012, Fu2022,LiuYang2017,WangYang2022,Yang2019TAMS, Yang2020MZ, YangZheng2018, Zheng2019}, further enriches this framework. These developments highlight the broader landscape of scalar curvature geometry, connecting classical K\"ahler theory to  non-Kähler geometric analysis.\\

    In this paper, we generalize Calabi's energy functional to Hermitian manifolds and study the existence of extremal Hermitian metrics with constant Chern scalar curvature--a non-K\"ahler analogue of cscK metrics. Let $(M,\omega_g)$ be a compact Hermitian manifold with complex dimension $n\geq 2$. Let $\omega_g=\sq g_{i\bar j}d z^i\wedge d\bar z^j$ and $s_g$ be the Chern scalar curvature.  For any $p>1$, we define the \emph{$p$-Calabi functional}  as
\beq
    C_{p}(\omega_g)=\lrp{\int |s_g|^p\frac{\omega_g^n}{n!}}\lrp{\int\frac{\omega_g^n}{n!}}^{-\frac{n-p}{n}}.
\eeq
This functional is normalized to ensure scale invariance: $ C_{p}(\lambda\omega_g)=C_p(\omega_g)$ for $\lambda\in \R^+$.
A Hermitian metric $\omega_g$ is called a \textit{$p$-conformal extremal metric} if  it is a critical point of  $C_p$ within its conformal class, satisfying
    \begin{align}
    \pdd C_{p}(e^{tf}\omega_g)=0
    \end{align}
    for any $ f\in C^\infty(M,\R)$. Our first main result establishes the Euler-Lagrange equation characterizing these critical metrics, generalizing Calabi's extremal condition to non-K\"ahler geometry:

\begin{thm}\label{thm1}
    Let $(M,\omega_g)$ be a compact Hermitian manifold. Then $\omega_g$ is $p$-conformal extremal
    if and only if the following equation holds in the sense of distributions
    \beq
    \square_g^*\left(s_g|s_g|^{p-2}\right)=\frac{n-p}{np}\left(|s_g|^p-\intbar|s_g|^p\frac{\omega_g^n}{n!}\right)\label{pextremal}
    \eeq
    where $\square_g^*$ is the formal adjoint of the operator $\square_g=\mathrm{tr}_{\omega_g}\sq\p\bar\p$ on smooth functions.
\end{thm}

\noindent It is straightforward to verify that equation  \eqref{pextremal} constitutes a fourth-order nonlinear partial differential equation governing the metric tensor
$g$ within the framework of Riemannian geometry. However, the current literature offers relatively few comprehensive results on the existence, uniqueness, and regularity of solutions for such equations.  Motivated by these challenges, we restrict our attention to the case $p=n$ where the $n$-conformal extremal metric satisfies a homogeneous equation
    \beq
\square_g^*(s_g|s_g|^{n-2})=0.\label{nextremal}
\eeq This structural simplification permits the application of elliptic regularity theory and geometric energy methods, while the homogeneity condition eliminates certain obstructions arising from lower-order terms, thereby enabling a more tractable analysis of the system’s global behavior.Consequently, we establish that equation \eqref{nextremal} has a unique solution in the conformal class of any Hermitian structure.

    \begin{thm}\label{thm2}
        Let $(M,\omega_g)$ be a compact Hermitian manifold. Then there exists a unique (up to scaling) smooth Hermitian metric $\omega_E$ in the conformal class of $\omega_g$ such that $\omega_E$ is an $n$-conformal extremal metric, i.e., $\omega_E$ satisfies equation \eqref{nextremal}. Moreover, $\omega_E$ minimizes the functional $C_n$ in the conformal class of $\omega_g$, i.e.,
        \beq C_n(\omega_E)=\inf_{f\in C^\infty(M,\R)}C_n(e^f \omega_g). \eeq
            \end{thm}

        \noindent       Crucially, the fourth-order system \eqref{nextremal} decouples into two compatible second-order equations, whose structural compatibility ensures solvability by elliptic PDE theory.  Within any conformal class
        $\omega_g$, Gauduchon's theorem guarantees a unique metric $\omega_G$ (up to scaling) satisfying $\p\bp\omega_G^{n-1}=0$.  The extremal metric
        $\omega_E$
        solving \eqref{nextremal} and the Gauduchon metric $\omega_G$
        generally diverge; their distinction arises from the former's dependence on fourth-order equation versus the latter's second-order elliptic PDE. A concrete example illustrating this non-identity is provided below.

\begin{example} Let $M=\C\P^2\#\bar{2\C\P^2}$. It is well-known that $M$ has $c_1(M)>0$ and does not admit K\"ahler-Einstein metrics.  Let $\omega_g\in c_1(M)$ be a K\"ahler metric.  Let $\omega_E$ be the Hermitian metric minimizes the $n$-conformal Calabi functional in the conformal class of $\omega_g$. Then $\omega_E$ is not a Gauduchon metric. Otherwise, it is easy to see that $\omega_E=\omega_g$ up to scaling. Moreover, it satisfies
    \beq \square_g^*s_g=0.\eeq
    Since $\omega_g$ is K\"ahler, by maximum principle, $s_g$ is constant. Since $[\omega_g]=c_1(M)$, $\omega_g$ is actually K\"ahler-Einstein. This is absurd.
\end{example}

\noindent
Furthermore, we establish a sharp criterion for
 $\omega_E=\omega_G$: equality holds precisely when the scalar curvature of
$\omega_E$ is (nonzero) constant.

\btheorem\label{thm3} Let $(M,\omega_g)$ be a compact Hermitian manifold and $\omega_E$ be the $n$-conformal extremal metric in the conformal class of $\omega_g$.
\bd \item If  $\omega_E$ is Gauduchon, then it has constant Chern scalar curvature.
\item If  $\omega_E$ has constant nonzero Chern scalar curvature, then $\omega_E$ is Gauduchon.
\ed
\etheorem

\noindent Observe that if  $\omega_g$
is a Hermitian metric with vanishing Chern scalar curvature, it satisfies
$C_n(\omega_g)=0$,  thereby globally minimizing the functional $C_n$. Such a metric constitutes an $n$-conformal extremal metric. However, this metric fundamentally differs from a Gauduchon metric $\omega_G$.

\begin{example} Let $M_1=\C\P^2\#\bar{2\C\P^2}$ and $\omega_1\in c_1(M_1)$ be a K\"ahler metric.  Let $M_2=\Sigma_2$ be a Riemann surface with genus $2$ and $\omega_2$ be the canonical metric with $\mathrm{Ric}(\omega_2)=-\omega_2$. Suppose that $M=M_1\times M_2$ and $\omega_g=2\omega_1+\omega_2$ which is a K\"ahler metric in the class $2c_1(M_1)-c_1(M_2)$. A straightforward computation shows that
    \beq \int_M c_1(M)\wedge \omega_g^2=0. \eeq
    One can see that there exists some $f\in C^\infty(M,\R)$ such that $\omega_E= e^f \omega_g$ has zero Chern scalar curvature (\cite[Theorem~1.2]{WangYang2022}). We claim that $\omega_E$ can  not be Gauduchon. Otherwise, $\omega_E=\omega_g$ up to scaling and so $\omega_g=2\omega_1+\omega_2$ is a K\"ahler metric with zero scalar curvature. This implies that $\omega_1$ has positive constant scalar curvature. This is a contradiction.
\end{example}

While the existence, uniqueness, and regularity of solutions to the $p$-conformal extremal equation \eqref{pextremal} remain unresolved, we establish certain results analogous to those in Theorem \ref{thm3}.

    \begin{thm}\label{thm4}
    Let $(M,\omega_g)$ be a compact Hermitian manifold and $\omega_p$ be a $p$-conformal extremal metric in the conformal class of $\omega_g$ with $p\geq 2$.
    \bd \item
    If $\omega_p$ is Gauduchon and $(n-p)s_p\geq 0$, then it has constant Chern scalar curvature.
    \item
    If $\omega_p$ has constant nonzero Chern scalar curvature, then $\omega_p$ is Gauduchon.
    \ed

\end{thm}

\noindent We also obtain a  criterion for $p$-conformal extremal metrics, which is analogous to the classical result in the K\"ahler setting.
\bcorollary \label{cor1} Let $(M,\omega_g)$ be a compact Hermitian manifold. If $\omega_g$ is Gauduchon and has constant Chern scalar curvature, then it is $p$-conformal extremal for any $p>1$.
\ecorollary

 \noindent \textbf{Acknowledgements.} The first named  author thanks  Valentino Tosatti and Shing-Tung Yau for insightful comments.

\vskip 2\baselineskip

    \section{Preliminaries}

    Let $(M,\omega_g)$ be a compact Hermitian manifold with complex dimension $n$. For any $f\in C^\infty(M,\R)$, the complex Laplacian $\square_g f$ of $f$ is defined as
    \beq
    \square_g f:=\tr_{\omega_g}\left(\sq\p\bp f\right).
    \eeq

    \blemma \label{lem1}  Let $f\in C^\infty(M,\R)$.

    \bd \item One has  \beq  \bp^*(f\omega_g)=f\bp^*\omega_g+\sq\p f,\ \ \  \sq\la\bp f,\p^*\omega_g\ra=\tr_{\omega_g}\sq\p\bp f+\bp^*\bp f. \label{c1}\eeq

    \item  For any $(1,0)$ form $\eta$,  \beq \p^*(f\eta)=f\p^*\eta-\la \eta,\p
    f\ra.\label{c5}\eeq
    \ed

    \bproof For any smooth $(1,0)$-form $\eta\in \Gamma(M,T^{*1,0}M)$,
    we have the global inner product \be
    \left(\bp^*(f\omega_g),\eta\right)=\left(f\omega_g,\bp\eta\right)
    =\left(\omega_g,
    f\bp\eta\right)&=&\left(\omega_g,\bp(f\eta)\right)-\left(\omega_g,\bp f\wedge
    \eta\right)\\
    &=&\left(f\bp^*\omega_g,\eta\right)-\left(\omega_g,\bp f\wedge
    \eta\right)\\
    &=&\left(f\bp^*\omega_g,\eta\right)+\sq\left(\p f, \eta\right) \ee
    where the last identity follows from the fact that $f$ is real
    valued. The second part of $(1)$ can be proved in a similar way.  For any smooth function $\phi\in
    C^\infty(M,\R)$, we have  \be
    \left(\p^*(f\eta),\phi\right)&=&\left(f\eta,\p\phi\right)=\left(\eta,
    f\p\phi\right)=\left(\eta, \p(f\phi)-\p f\cdot \phi\right)\\
    &=&\left(f\p^*\eta,\phi\right)-\left(\la\eta, \p f\ra,
    \phi\right)\ee and we obtain (\ref{c5}). \eproof \elemma

    \begin{lemma}The formal adjoint $\square_g^*$ of $\square_g$ is given by
        \beq
        \square_g^*u=\left(-\sq\bp^*\p^*\omega_g\right)\cdot u-\Delta_d u-\square_g u, \label{adjoint}
        \eeq
        where $u\in C^\infty(M,\R)$ and  $\Delta_d u=d^*d u=\bp^*\bp u+\p^*\p u$.
    \end{lemma}
    \begin{proof} By Lemma \ref{lem1}, we have
        \beq \square_g u=\sq\la\bp u,\p^*\omega_g\ra-\bp^*\bp u.\eeq
        \noindent
        For any $v\in C^\infty(M,\R)$, one has
        \beq (\square_g u,v)=\left(\sq\la\bp u,\p^*\omega_g\ra-\bp^*\bp u,v\right)\\
        =\sq(\bp u,v\p^*\omega_g)-(u,\bp^*\bp v). \eeq
        Moreover,
        \beq \sq(\bp u,v\p^*\omega_g)=(u,-\sq\bp^*(v\p^*\omega_g))\eeq
        and by using Lemma \ref{lem1} again, we obtain
        \beq \bp^*(v\p^*\omega_g)=v\bp^*\p^*\omega_g -\la\p^*\omega_g,\bp v\ra=v\bp^*\p^*\omega_g-\sq(\tr_{\omega_g}\sq\p\bp v+\p^*\p v).\eeq
        This gives formula \eqref{adjoint}.
    \end{proof}

        \begin{lemma}\label{lem2.5}
        Let $(M,\omega_g)$ be a Hermitian manifold and  $f\in C^\infty(M,\R)$. Let $s_f$ be the Chern scalar curvature of $\omega_f=e^f\omega_g$. Then
        \beq\label{2.3}
        s_f=e^{-f}\lrp{s_g-n\square_g f}
        \eeq
        where $s_g$ is the Chern scalar curvature of $\omega_g$.
    \end{lemma}

    \begin{proof} It follows from a straightforward computation.
        Let $\Theta_f$ be the Chern curvature of $\omega_f$. Then one can compute that
        \beq
        \lrp{\Theta_f}_{i\bar jk\bar\l}=e^f\lrp{\Theta_{i\bar jk\bar\l}-g_{k\bar\l}\frac{\p^2f}{\p z^i\p\bar z^j}}
        \eeq
        and
        \beq
        s_f=h_f^{i\bar j}h_f^{k\bar\l}\lrp{\Theta_f}_{i\bar jk\bar\l}=e^{-f}\lrp{s_g-n\square_g f}
        \eeq
        where $h^{i\bar j}_f =e^{-f}g^{i\bar j}$.
    \end{proof}

        \begin{lemma}
        Let $(M,\omega_g)$ be a compact Hermitian manifold and $\omega_f=e^f\omega_g$  for some $f\in C^\infty(M,\R)$. Let $\square_f:=\mathrm{tr}_{\omega_f}\sq\p\bar\p$, then for any $u\in C^\infty(M,\R)$,
        \beq\label{2.5}
        \square_fu=e^{-f}\square_g u,\quad \square_f^*u=e^{-nf}\square_g^*(e^{(n-1)f}u).
        \eeq
    \end{lemma}

    \begin{proof} Indeed, we have
        \beq
        \square_fu=\mathrm{tr}_{\omega_f}\sq\p\bar\p u=e^{-f}\square_g u.
        \eeq
        For any $v\in C^\infty(M,\R)$,
        \be
        \int v\cdot \square_f^*u\cdot\frac{\omega_f^n}{n!}=\int u \cdot \square_f v\cdot\frac{\omega_f^n}{n!}
        &=&\int ue^{(n-1)f} \square_g v\cdot\frac{\omega_g^n}{n!}\\
        &=&\int ve^{-nf}\square_g^*(e^{(n-1)f}u)\cdot\frac{\omega_f^n}{n!}.
        \ee
    Hence, $\square_f^*u=e^{-nf}\square_g^*(e^{(n-1)f}u)$.  \end{proof}

\vskip 2\baselineskip

    \section{Proofs of main theorems}
    In this section, we prove Theorem \ref{thm1}, Theorem \ref{thm2} and Theorem \ref{thm3}.
    \begin{proof}[Proof of Theorem \ref{thm1}.] By using    Lemma \ref{lem2.5}, we have
    \beq
    C_p(e^{tf}\omega_g)=\lrp{\int e^{(n-p)tf}|s_g-nt\square_g f|^p\frac{\omega^n_g}{n!}}\lrp{\int e^{ntf}\frac{\omega_g^n}{n!}}^{-\frac{n-p}{n}}.
    \eeq
    When $p>1$, one can see clearly that $C_p(e^{tf}\omega_g)$ is differentiable with respect to $t$. A straightforward computation shows that
    \begin{align*}
    \frac{d}{dt} C_p(e^{tf}\omega_g)=A_1(t)+A_2(t)+A_3(t),
    \end{align*}
    where $A_1(t)$, $A_2(t)$ and $A_3(t)$ are given by
    \begin{align*}
    A_1(t)=&(n-p)\lrp{\int fe^{(n-p)tf}\left|s_g-nt\square_g f\right|^p\frac{\omega^n_g}{n!}}\lrp{\int e^{ntf}\frac{\omega_g^n}{n!}}^{-\frac{n-p}{n}},\\
    A_2(t)=&-np\lrp{\int e^{(n-p)tf}\cdot \square_g f\cdot\left(s_g-nt\square_g f\right) \cdot \left|s_g-nt\square_g f\right|^{p-2}\frac{\omega^n_g}{n!}}\lrp{\int e^{ntf}\frac{\omega_g^n}{n!}}^{-\frac{n-p}{n}},\\
    A_3(t)=&-(n-p)\lrp{\int e^{(n-p)tf}\left|s_g-nt\square_g f\right|^p\frac{\omega^n_g}{n!}}\lrp{\int fe^{ntf}\frac{\omega_g^n}{n!}}\lrp{\int e^{ntf}\frac{\omega_g^n}{n!}}^{-\frac{n-p}{n}-1}.
    \end{align*}
    When $t=0$, one has
    \begin{align*}
    A_1(0)=&(n-p)\lrp{\int f|s_g|^p\frac{\omega^n_g}{n!}}\lrp{\int\frac{\omega_g^n}{n!}}^{-\frac{n-p}{n}},\\
    A_2(0)=&-np\lrp{\int \square_g f\cdot s_g\cdot |s_g|^{p-2}\frac{\omega^n_g}{n!}}\lrp{\int\frac{\omega_g^n}{n!}}^{-\frac{n-p}{n}},\\
    A_3(0)=&-(n-p)\lrp{\int|s_g|^p\frac{\omega^n_g}{n!}}\lrp{\int f\frac{\omega_g^n}{n!}}\lrp{\int\frac{\omega_g^n}{n!}}^{-\frac{n-p}{n}-1}\\
    =&-(n-p)\lrp{\intbar|s_g|^p\frac{\omega^n_g}{n!}}\lrp{\int f\frac{\omega_g^n}{n!}}\lrp{\int\frac{\omega_g^n}{n!}}^{-\frac{n-p}{n}}.
    \end{align*}
Therefore one obtains
    \begin{align*}
    \pdd C_p(e^{tf}\omega_g)=&\lrp{\int F\frac{\omega^n_g}{n!}}\lrp{\int \frac{\omega_g^n}{n!}}^{-\frac{n-p}{n}},
    \end{align*}
    where
    \beq F= (n-p)f\lrp{|s_g|^p-\intbar|s_g|^p\frac{\omega^n_g}{n!}}-np\square_g f\cdot s_g\cdot |s_g|^{p-2}.\eeq
    Since $p>1$, the function $s_g\cdot |s_g|^{p-2}$ is continuous on $M$. Hence, we obtain the Euler-Lagrange equation in the sense of distributions
    \beq
    \square_g^*\left(s_g|s_g|^{p-2}\right)=\frac{n-p}{np}\left(|s_g|^p-\intbar|s_g|^p\frac{\omega_g^n}{n!}\right).
    \eeq
        This completes the proof of Theorem \ref{thm1}.
\end{proof}

\vskip 1\baselineskip

    \begin{proof}[Proof of Theorem \ref{thm2}.]  Let $[\omega_g]$ be the conformal class of $\omega_g$.  It is well-known that there exists a unique (up to scaling) Gauduchon metric $\omega_G\in[\omega_g]$, i.e., $\p\bp\omega_G^{n-1}=0$. Let $\omega_f=e^f\omega_G$ for  $f\in C^\infty(M,\R)$. Let $\square_G =\tr_{\omega_G}\sq\p\bar\p$ and $\square_f =\tr_{\omega_f}\sq\p\bar\p$. Let $s_G$ be the Chern scalar curvature of $\omega_G$ and $s_f$ be the Chern scalar curvature of $\omega_f$.     By Theorem \ref{thm1},  $\omega_f$ is $n$-conformal extremal if and only if
        \beq\label{3.4}
            \square_f^*(s_f|s_f|^{n-2})=0.
        \eeq We shall show that there exists a unique (up to additive constants) function $f\in C^\infty(M,\R)$ satisfying (\ref{3.4}). Indeed, by using formula
    (\ref{2.5}), we obtain
        \beq\label{3.51}
        \square_f^*\left(s_f|s_f|^{n-2}\right)=e^{-nf}\square_G^*\left(e^{(n-1)f}s_f|s_f|^{n-2}\right).
        \eeq
        By \eqref{2.3}, one has $s_f=e^{-f}\lrp{s_G-n\square_G f}$ and so
        \beq\label{3.5}
        \square_f^*\left(s_f|s_f|^{n-2}\right)=e^{-nf}\square_G^*\left(\left(s_G-n\square_Gf\right)\left|s_G-n\square_Gf\right|^{n-2}\right)=0.
        \eeq
 When $\omega_G$ is Gauduchon,  by formula \eqref{adjoint}, equation \eqref{3.5} is equivalent to
 \beq \left(\Delta_{d,G}+\square_G\right)\left(\left(s_G-n\square_Gf\right)\left|s_G-n\square_Gf\right|^{n-2}\right)=0. \eeq
 Let $\phi= \left(s_G-n\square_Gf\right)\left|s_G-n\square_Gf\right|^{n-2}.$
 We shall show that $\left(\Delta_{d,G}+\square_G\right)\phi=0$ if and only if $\phi$ is a constant. Indeed, one can see clearly that
 \be \int\phi \square_G\phi \frac{\omega_G^{n}}{n!}&=&\frac{1}{2}\left(\int\sq \p\bp\phi^2\wedge  \frac{\omega_G^{n-1}}{(n-1)!}-\int 2\sq \p\phi\wedge\bp\phi\wedge \frac{\omega_G^{n-1}}{(n-1)!} \right)\\&=&-\int |\p\phi|^2 \frac{\omega_G^{n}}{n!}. \ee
    Hence
\beq 0= \int\phi  \left(\Delta_{d,G}\phi+\square_G\phi\right)  \frac{\omega_G^{n}}{n!}=\int |\bp\phi|^2 \frac{\omega_G^{n}}{n!}.\eeq
    Therefore, $\phi$ is constant and there exists some  $C\in\R$ such that
        \beq\label{3.6}
            s_G-n\square_Gf=C.
        \eeq
    Since $\omega_G$ is Gauduchon,  it is easy to see that
        \beq C=\intbar s_G  \frac{\omega_G^{n}}{n!}.\eeq
    Now equation (\ref{3.6}) can be written as
        \begin{align}\label{3.8}
            \square_Gf=\psi,\quad\text{with}\quad\psi=\frac1n\lrp{s_G-C}.
        \end{align}
        Since $\int\psi\cdot\omega_G^n/n!=0$, equation  (\ref{3.8}) has a unique solution (up to additive constants) $f\in C^\infty(M,\R)$. Thereofore, there exists a unique (up to scaling) $n$-conformal extremal metric $\omega_f=e^f\omega_G$ in the conformal class $[\omega_g]$.\\

         On the other hand, it is easy to compute that
        \beq
            \frac{d^2}{dt^2}C_n(e^{tf}\omega_g)=n^3(n-1)\int \left(\square_g f\right)^2\left|s_g-nt\square_g f\right|^{n-2}\frac{\omega_g^n}{n!}\geq 0.
        \eeq
        Hence, if $\omega_E$ is a critical point of the functional
        $C_n$, then it is a minimizer.
    \end{proof}

\vskip 1\baselineskip

    \begin{proof}[Proof of Theorem \ref{thm3}.]   Let $s_E$ be the Chern scalar curvature of  the $n$-conformal extremal metric $\omega_E$.
        \bd \item  If $\omega_E$ is Gauduchon, by using formula \eqref{adjoint} equation \eqref{nextremal} is reduced to
        \beq \left(\Delta_{d,E}+\square_E\right)\left(s_E|s_E|^{n-2}\right)=0.\eeq
    By using a similar maximum principle as shown in the proof of Theorem \ref{thm1},   $s_E|s_E|^{n-2}$ is constant and so $s_E$ is constant.

    \item If $s_E$ has constant nonzero Chern scalar curvature, then by formula \eqref{adjoint} $$0=\square_E^*\left(s_E|s_E|^{n-2}\right)=-\left(\sq\bp^*\p^*\omega_E\right)\cdot \left(s_E|s_E|^{n-2}\right),$$
    where  dual operators are taken with respect to $\omega_E$.
     Hence, $\sq\bp^*\p^*\omega_E=0$, which is equivalent to  $\p\bp\omega^{n-1}_E=0$, i.e.,  $\omega_E$ is Gauduchon. \ed
     This completes the proof of Theorem \ref{thm3}.
    \end{proof}

Let $(M,\omega_g)$ be a compact Hermitian manifold. We denote $\mathscr F(\omega_g)$ by the \textit{total scalar curvature} of $\omega_g$, i.e.,
\beq
\mathscr F(\omega_g)=\int s_g\omega_g^n.
\eeq

\bcorollary Let $(M,\omega_g)$ be a compact Hermitian manifold. Suppose that $\omega_G$ and $\omega_E$ are the Gauduchon metric and  $n$-conformal extremal metric  in the conformal class of $\omega_g$ respectively.
\bd \item $\omega_E$ has positive Chern scalar curvature if and only if  $\mathscr F(\omega_G)>0$;
\item $\omega_E$ has negative Chern scalar curvature if and only if  $\mathscr F(\omega_G)<0$;
\item $\omega_E$ has  zero Chern scalar curvature if and only if $\mathscr F(\omega_G)=0$.
\ed
\ecorollary

    \begin{proof}
        From the proof of Theorem \ref{thm1}, we know that $\omega_E=e^f\omega_G$ where $\omega_G$ is the Gauduchon metric and $f\in C^\infty(M,\R)$ satisfies
        \begin{align}
            \square_Gf=\vphi,\quad\text{with}\quad\vphi=\frac1n\lrp{s_G-\frac{\mathscr F(\omega_G)}{\int\omega_G^n}}.
        \end{align}
        Thus, the Chern scalar $s_E$ of $\omega_E$ satisfies
        \beq
            s_E=e^{-f}(s_G-n\square_G f)=\frac{e^{-f}}{\int\omega_G^n}\cdot \mathscr F(\omega_G).
        \eeq
        Hence, the conclusions follow.
    \end{proof}

\vskip 2\baselineskip

\section{Some consequences of $p$-conformal extremal metrics}
In this section we prove Theorem \ref{thm4} and Corollary \ref{cor1}.

    \begin{proof}[Proof of Theorem \ref{thm4}.]
         Recall that a $p$-conformal extremal metric  satisfies the following equation in the sense of distributions:
         \beq   \square_p^*\left(s_p|s_p|^{p-2}\right)=\frac{n-p}{np}\left(|s_p|^p-\intbar|s_p|^p\frac{\omega_p^n}{n!}\right)\eeq
        where $\square_p^*$ is the formal adjoint of $\square_p=\mathrm{tr}_{\omega_p}\sq\p\bp$ with respect to $\omega_p$. Multipling by $|s_p|^p$ and integrating, one has
    \beq \frac{n-p}{np}\int\lrp{|s_p|^p-\intbar|s_p|^p\frac{\omega_p^n}{n!}}^2\frac{\omega_p^n}{n!}=\int|s_p|^p\square_p^*\left(s_p|s_p|^{p-2}\right)\frac{\omega_p^n}{n!}, \label{4.1}\eeq
    and the right hand reads
\beq \int|s_p|^p\square_p^*\lrp{s_p|s_p|^{p-2}}\frac{\omega_p^n}{n!}=\int \square_p\left(|s_p|^p\right)\lrp{s_p|s_p|^{p-2}}\frac{\omega_p^n}{n!}.\eeq
A straightforward computation shows
    \beq \int|s_p|^p\square_p^*\lrp{s_p|s_p|^{p-2}}\frac{\omega_p^n}{n!}
            =p\int\lrp{|s_p|^{2p-2}\square_ps_p+(p-1)s_p|s_p|^{2p-4}|\p s_p|^2}\frac{\omega_p^n}{n!}.\label{4.2}
    \eeq

    \noindent For part $(1)$,  if $p=n$, it follows from Theorem \ref{thm3}. When $p\neq n$, since $\omega_p$ is Gauduchon, one has
    \beq\int\square_p\lrp{s_p|s_p|^{2p-2}}\frac{\omega_p^n}{n!}=0.\eeq
    On the other hand, we compute that  $$
        \int\square_p\lrp{s_p|s_p|^{2p-2}}\frac{\omega_p^n}{n!}=(2p-1)\int\lrp{|s_p|^{2p-2}\square_ps_p+2(p-1)s_p|s_p|^{2p-4}|\p s_p|^2}\frac{\omega_p^n}{n!}.
    $$
    In particular, we have
    \beq \int |s_p|^{2p-2}\square_ps_p\frac{\omega_p^n}{n!}=-2(p-1)\int s_p|s_p|^{2p-4}|\p s_p|^2\frac{\omega_p^n}{n!}.\eeq
        \noindent
    By  \eqref{4.1} and \eqref{4.2}, we obtain
        \be
            \frac{(n-p)^2}{np}\int\lrp{|s_p|^p-\intbar|s_p|^p\frac{\omega_p^n}{n!}}^2\frac{\omega_p^n}{n!}=-p(p-1)(n-p)\int s_p|s_p|^{2p-2}|\p s_p|^2\frac{\omega_p^n}{n!}.
        \ee
        Since $(n-p)s_p\geq 0$, we deduce that the right hand side is non-positive. Hence $s_p$ is a constant.\\

    \noindent   For $(2)$, if $s_p$ has constant nonzero Chern scalar curvature, then \eqref{pextremal} is reduced to
    \beq \square_p^*\left(s_p|s_p|^{p-2}\right)=0 \eeq
    and by  formula \eqref{adjoint}, we also have  $$\square_p^*\left(s_p|s_p|^{p-2}\right)=-\left(\sq\bp^*\p^*\omega_p\right)\cdot \left(s_p|s_p|^{p-2}\right),$$
    where  dual operators are taken with respect to $\omega_p$. Hence, $\sq\bp^*\p^*\omega_p=0$, i.e.,  $\omega_p$ is Gauduchon.
    \end{proof}

\vskip 1\baselineskip

    \begin{proof}[Proof of Corollary \ref{cor1}.] If $\omega_g$ is Gauduchon and  $s_g$ is constant, then for any $p>1$,
    \be
    \square_g^*\left(s_g|s_g|^{p-2}\right)=\frac{n-p}{np}\left(|s_g|^p-\intbar|s_g|^p\frac{\omega_g^n}{n!}\right)=0.
    \ee
    Hence $\omega_g$ is a $p$-conformal extremal metric.
    \end{proof}

\vskip 2\baselineskip


\begin{thebibliography}{99}

    \bibitem{Angella2017}
    Daniele Angella, Simone Calamai, and Cristiano Spotti, On the Chern-Yamabe problem.
    {\em Math. Res. Lett.}
    \textbf{24} (2017), no.~3, 645--677.

    \bibitem{Angella2020}
    Daniele Angella, Simone Calamai, and Cristiano Spotti,  Remarks on Chern-Einstein Hermitian metrics.
    {\em Math. Z.}
    \textbf{295} (2020), 1707--1722.

    \bibitem{ACGC2008}
    Vestislav Apostolov, David M.J. Calderbank, Paul Gauduchon, and Christina Tønnesen-Friedman,
    Hamiltonian 2-forms in Kähler geometry. III. Extremal metrics and stability.
    \textit{Invent. Math.}
    \textbf{173} (2008), no.~3, 547--601.


    \bibitem{Apostolov1999}
    Vestislav Apostolov and Tedi Dr\u{a}ghici,
    Hermitian conformal classes and almost {K}\"{a}hler structures on {$4$}-manifolds.
    {\em Differential Geom. Appl.}
    \textbf{11} (1999), no.~2,179--195.

    \bibitem{Apostolov1996}
    Vestislav Apostolov, Johann Davidov, and Oleg Mu\v{s}karov,
    Compact self-dual Hermitian surfaces.
    {\em Trans. Amer. Math. Soc.}
    \textbf{348} (1996), no.~8, 3051--3063.


    \bibitem{AP09}
    Claudio Arezzo and Frank Pacard,
    Blowing up K\"ahler manifolds with constant scalar curvature {II}.
    {\it Ann. of Math.}\textbf{170} (2009),no.~2,  685--738.

    \bibitem{BL23}
    Giuseppe Barbaro and Mehdi  Lejmi,
    Second Chern-Einstein metrics on four-dimensional almost-Hermitian manifolds.
    \emph{Complex Manifolds}
    \textbf{10} (2023), no.~1, Paper No. 20220150.

    \bibitem{Broder2023}
    Kyle Broder and James Stanfield, On the Gauduchon curvature of Hermitian manifolds.
    {\em Internat. J. Math.},
    \textbf{34}(2023),no.~7, Paper No. 2350039.

    \bibitem{Calabi1982}
    Eugenio Calabi,
    Extremal K\"ahler metrics,
    {\it Seminar on Differential Geometry},
    \textbf{102}(1982), Princeton Univ. Press, Princeton, NJ, 259--290.

    \bibitem{Calabi1985}
    Eugenio Calabi,
    Extremal K\"ahler metrics, II.
    {\it Differential geometry and Complex analysis}
    (1985) Springer, 96--114.

\bibitem{Chen2009}
    Xiuxiong Chen, Space of K\"ahler metrics (III)--Lower bound of the Calabi enegery. \emph{Invent. Math.} \textbf{175}(2009), no.~3, 453--503.




\bibitem{Chen-Cheng2021a}
Xiuxiong Chen and Jingrui Cheng,
On the constant scalar curvature K\"ahler metrics (I)---A priori estimates,
\textit{J. Amer. Math. Soc.}
{\bf 34} (2021), no.~4, 909--936.

\bibitem{Chen-Cheng2021b}
Xiuxiong Chen and Jingrui Cheng,
On the constant scalar curvature Kähler metrics (II)—Existence results.
\textit{J. Amer. Math. Soc.}
{\bf 34} (2021), no.~4, 937--1009.

    \bibitem{CDS15}
Xiuxiong Chen, Simon Donaldson and Song Sun,
K\"ahler-Einstein metrics on Fano manifolds. III: Limits as cone angle approaches $2\pi$ and completion of the main proof.
{\it J. Amer. Math. Soc.}
{\bf 28}(2015), 235--278.


    \bibitem{ChenHe2008}
    Xiuxiong Chen and Weiyong He,
    On the Calabi flow.
    \textit{Amer. J. Math.}
    \textbf{130} (2008), no.~2, 539--570.

    \bibitem{Chen-Sun2014}
    Xiuxiong Chen and Song Sun,
    Calabi flow, geodesic rays, and uniqueness of constant scalar curvature K\"ahler metrics.
    \textit{Ann. of Math.}
    \textbf{180} (2014), no.~2, 407--454.


    \bibitem{Don99}
    Simon Donaldson,
    Symmetric spaces, K\"ahler geometry and Hamiltonian dynamics.
    {\it Amer. Math. Soc. Transl. Ser.} Northern California Symplectic Geometry Seminar.
    {\bf 196}(1999),no.~2,13--33.

    \bibitem{Don01}
    Simon Donaldson,
    Scalar curvature and projective embeddings, {I},
    {\it J. Differential Geom.}
    \textbf{59} (2001), 479--522.


    \bibitem{Donaldson2008}
    Simon Donaldson,
    Extremal metrics on toric surfaces: a continuity method.
    \textit{J. Differential Geom.}
    \textbf{79} (2008), no.~3, 389--432.





    \bibitem{Fu2012}
    Jixiang Fu,
    \newblock Specific non-K\"{a}hler Hermitian metrics on compact complex
    manifolds.
    \newblock {\em Recent developments in geometry and analysis},
    \textbf{23}(2012), Adv. Lect. Math. Int. Press, Somerville, MA,79--90.


    \bibitem{Fu2022}
    Jixiang Fu and Xianchao Zhou,
    \newblock Scalar curvatures in almost {H}ermitian geometry and some
    applications.
    \newblock {\em Sci. China Math.},
    \textbf{65}(2022), no.~12, 2583--2600.

    \bibitem{Futaki1983}
    Akito Futaki,
    An obstruction to the existence of Einstein K\"ahler metrics.
    {\it Invent. Math.}
    {\bf 73} (1983), no.~3, 437--443.


    \bibitem{Futaki1988}
    Akito Futaki,
    {\it K\"ahler-Einstein metrics and integral invariants},
    Lecture Notes in Mathematics,
    {\bf 1314}(1988), Springer, Berlin.




    \bibitem{LS94}
    Claude LeBrun and Santiago R Simanca,
    Extremal {K}\"ahler metrics and complex deformation theory.
    {\it Geom. and Func. Anal.}
    \textbf{4} (1994),  298--336.


    \bibitem{Lejmi2018}
    Mehdi Lejmi and Ali Maalaoui,
    \newblock On the Chern-Yamabe flow.
    \newblock {\em J. Geom. Anal.},
    {\bf 28}(2018), no.~3,2692--2706.



    \bibitem{LWZ2018}
    Haozhao Li, Bing Wang, and Kai Zheng,
    Regularity scales and convergence of the Calabi flow.
    {\it J. Geom. Anal.}
    {\bf 28} (2018), no.~3, 2050--2101.


    \bibitem{LiuYang2017}
    Kefeng Liu and Xiaokui Yang,
    \newblock Ricci curvatures on Hermitian manifolds.
    \newblock {\em Trans. Amer. Math. Soc.},
    {\bf 369}(2017), no.~7,5157--5196.


    \bibitem{Ross06}
    Julius Ross,
    Unstable products of smooth curves.
    {\it Invent. Math.} {\bf 165}(2006), 153--162.

    \bibitem{RT06}
    Julius Ross snd Richard Thomas,
    An obstruction to the existence of constant scalar curvature K\"ahler metrics.
    {\it J. Differential Geom.}
    \textbf{72} (2006), 429--466.


    \bibitem{Stoppa08}
    Jacopo Stoppa,
    K-stability of constant scalar curvature K\"ahler
  manifolds.
    {\it Adv. Math}. \textbf{221} (2009),  1397--1408.

    \bibitem{Szekelyhidi2007}
    G\'abor  Sz\'ekelyhidi, Extremal metrics and $K$-stability.
    \emph{Bull. Lond. Math. Soc.}
    {\bf 39} (2007), no.~1, 76--84.

    \bibitem{Szekelyhidi2014}
    G\'abor  Sz\'ekelyhidi,
    {\it An introduction to extremal K\"ahler metrics}.
     Grad. Stud. Math., 152 American Mathematical Society, Providence, RI, 2014. xvi+192 pp.

    \bibitem{Tian2002}
    Gang Tian,
    Extremal metrics and geometric stability.
    {\it Houston J. Math.}
    {\bf 28}(2002), no.~2, 411--432.


\bibitem{TW2007}
    Valentino Tosatti and Ben Weinkove,
    The Calabi flow with small initial energy.
    \emph{Math. Res. Lett.}
    \textbf{14}(2007), no.~6, 1033--1039.


    \bibitem{WangYang2022}
    Jun Wang and Xiaokui Yang,
    RC-positivity and scalar-flat metrics on ruled surfaces.
    \emph{Math. Z.}
    \textbf{301}(2022), 917--934.






    \bibitem{Yang2019TAMS}
    Xiaokui Yang,
    \newblock Scalar curvature on compact complex manifolds.
    \newblock {\em Trans. Amer. Math. Soc.},
    {\bf 371}(2019), no.~3,2073--2087.

    \bibitem{Yang2020MZ}
    Xiaokui Yang,
    \newblock Scalar curvature, Kodaira dimension and $\hat {A}$-genus.
    \newblock {\em  Math Z.},
    {\bf 295}(2020), 365--380.




    \bibitem{YangZheng2018}
    Bo Yang and Fangyang Zheng,
    \newblock On curvature tensors of Hermitian manifolds.
    \newblock {\em Comm. Anal. Geom.},
    {\bf 26}(2018), no.~5,1195--1222.





    \bibitem{Zheng2019}
    Fangyang Zheng,
    \newblock Some recent progress in non-{K}\"{a}hler geometry.
    \newblock {\em Sci. China Math.},
    {\bf 62}(2019),no.~11,2423--2434.


\end{thebibliography}

\end{document}